\documentclass[11pt,a4paper]{article}
\usepackage[T1]{fontenc}
\usepackage[utf8]{inputenc}
\usepackage[english]{babel}
\usepackage{lmodern,microtype}
\usepackage{amsmath,amssymb,amsthm,mathtools}
\usepackage[a4paper,margin=2.35cm,headheight=15pt]{geometry}
\usepackage{enumitem}
\usepackage{xcolor}
\usepackage[hidelinks]{hyperref}
\usepackage{booktabs,needspace}
\definecolor{navy}{HTML}{000000}
\hypersetup{colorlinks=true,linkcolor=navy,citecolor=navy,urlcolor=navy}
\newtheorem{theorem}{Theorem}[section]
\newtheorem{proposition}[theorem]{Proposition}
\newtheorem{lemma}[theorem]{Lemma}
\newtheorem{corollary}[theorem]{Corollary}
\newtheorem{example}[theorem]{Example}
\theoremstyle{remark}
\newtheorem{remark}[theorem]{Remark}

\newcommand{\B}{\mathcal B}
\newcommand{\Ran}{\mathcal R}
\newcommand{\Ker}{\mathcal N}
\newcommand{\cc}{\mathcal K}
\newcommand{\qn}{\mathcal H_0}

\newcommand{\spn}{\operatorname{span}}

\setlist[enumerate]{leftmargin=2em,itemsep=.35em,topsep=.4em}
\title{On Mbekhta's conjecture and pseudo-Fredholm\\ operators on Banach and Hilbert spaces}
\author{Zakariae Aznay \qquad Fadil Chabbabi\\[1mm]
\small Laboratory (A.F.N.L.A), Department of Mathematics, Faculty of Sciences,\\
\small Abdelmalek Essa\^adi University, Tetouan, Morocco\\[1mm]
\small \texttt{aznayzakariae@gmail.com}\qquad \texttt{f.chabbabi@uae.ac.ma}}
\date{}
\begin{document}
\maketitle
\begin{abstract}
We give negative answers to Mbekhta's conjecture \cite[p.~261]{Mbekhta} and to a related conjecture on pseudo-B-Fredholm operators \cite[p.~202]{AZO}. A sufficient condition for pseudo-Fredholmness on Banach spaces is established. It yields a characterization on Hilbert spaces in terms of the closedness of $\cc(T)$, $\cc(T)+\qn(T)$ and $\Ran(T)+\qn(T)$, as well as characterizations of pseudo-B-Fredholm and pseudo-B-Weyl operators on Banach spaces. A counterexample on a separable Banach space shows that the three closedness conditions do not characterize pseudo-Fredholmness in general. We also construct two operators having the same analytic core, quasinilpotent part, and kernels and ranges of all powers, although only one is pseudo-Fredholm. Moreover, a semi-B-Fredholm operator on a Banach space need not be a direct sum of a semi-Fredholm operator and a nilpotent operator, answering the question raised in \cite[p.~113]{AznayKato2023}. Finally, the three closedness conditions in the Hilbert-space characterization are shown to be independent.
\end{abstract}

\section{Introduction and preliminaries}
Let $X$ be a complex Banach space and $\mathcal B(X)$ the Banach algebra of all bounded linear operators on $X$. For $T\in\mathcal B(X)$, we denote by $\Ker(T)$ and $\Ran(T)$ its kernel and range, respectively. A subspace $M$ of $X$ is called $T$-invariant if $T(M)\subseteq M$. We write $(M,N)\in Red(T)$ if $M$ and $N$ are closed $T$-invariant subspaces such that $X=M\oplus N$. For a closed $T$-invariant subspace $M$, we denote by $T|_M$ the restriction of $T$ to $M$, viewed as an operator on $M$.

An operator $T$ is called pseudo-Fredholm if there exists $(M,N)\in Red(T)$ such that $T|_M$ is semi-regular and $T|_N$ is quasinilpotent \cite{Mbekhta}. Recall that an operator $S$ is semi-regular if $\Ran(S)$ is closed and $\Ker(S)\subseteq \Ran^\infty(S)$, where $\Ran^\infty(S)=\bigcap_{n\geq 1}\Ran(S^n)$. This decomposition is known as a generalized Kato decomposition. When $T|_N$ is nilpotent, it is called a Kato-type decomposition. Recall that $Q$ is quasinilpotent if $\lim_{n\to\infty}\|Q^n\|^{1/n}=0$. An operator is Fredholm if its range is closed and its kernel and cokernel are finite-dimensional; it is Weyl if it is Fredholm with index zero. Moreover, $T$ is said to be pseudo-B-Fredholm if it can be written as $T=F\oplus Q$, where $F$ is Fredholm and $Q$ is quasinilpotent. If $F$ is a Weyl operator, then $T$ is called pseudo-B-Weyl. An operator is called pseudo semi-B-Fredholm (respectively, quasi semi-B-Fredholm) if it is the direct sum of a semi-Fredholm operator and a quasinilpotent (respectively, nilpotent) operator. For more details on these classes of operators, see \cite{AZO,AznayKato2023}.

The analytic core $\cc(T)$ and the quasinilpotent part $\qn(T)$ are defined by
\[
 \cc(T)=\left\{x\in X:\begin{array}{l}
 \exists c>0,\ \exists(x_n)_{n\ge0}\subset X,\ x_0=x,\\
 Tx_{n+1}=x_n,\quad\|x_n\|\le c^n\|x\|\quad(n\ge0)
 \end{array}\right\},
\]
and
\[
 \qn(T)=\{x\in X:\lim_{n\to\infty}\|T^nx\|^{1/n}=0\}.
\]
The following properties are well known (see \cite[Proposition~1.4]{Mbekhta}):
\begin{equation}\label{eq:corefacts}
 T \cc(T)=\cc(T),\qquad \cc(T)\subseteq \Ran^\infty(T),\qquad T^{-1}(\qn(T))=\qn(T).
\end{equation}
If $Q$ is quasinilpotent, then $\cc(Q)=\{0\}$ and $\qn(Q)=X$. Moreover, if $S$ is semi-regular, then
\[
 \cc(S)=\Ran^\infty(S)\text{ is closed},\qquad \qn(S)\subseteq \cc(S)\subseteq \Ran(S).
\]

In \cite[p.~261]{Mbekhta}, Mbekhta asked whether the closedness of $\Ran(T)+\qn(T)$ and $\cc(T)\cap \Ker(T)$ implies that $T$ is pseudo-Fredholm on a Hilbert space. We give a negative answer to this conjecture. More precisely, we show that a bilateral weighted shift already considered by Mbekhta in the same paper \cite[pp.~274--275]{Mbekhta} is injective and satisfies
\[
 \Ran(T)+\qn(T)=H,\qquad \cc(T)\cap \Ker(T)=\{0\},
 \qquad T\text{ is not pseudo-Fredholm}.
\]
In \cite[p.~202]{AZO}, Aznay, Ouahab and Zariouh proposed  a characterization of pseudo B-Fredholm operators by the finite dimensions of $\cc(T)\cap \Ker(T)$ and $X/(\Ran(T)+\qn(T))$. The same example shows that the sufficiency assertion of this conjecture is false, since both dimensions are zero.

We then investigate the role of these  subspaces in the characterization of pseudo-Fredholm operators. We first prove that $T\in\mathcal B(X)$ is pseudo-Fredholm whenever $\cc(T)$ and $\cc(T)+\qn(T)$ are closed and both $\Ker(T)\cap \cc(T)$ and $\Ran(T)+\qn(T)$ are complemented in $X$. Here a subspace is complemented if it is the range of a bounded projection. In particular, complemented subspaces are closed. Since every closed subspace of a Hilbert space is complemented, we obtain that, when $H$ is a complex Hilbert space, $T\in\mathcal B(H)$ is pseudo-Fredholm if and only if
\[
 \cc(T),\qquad \cc(T)+\qn(T),\qquad \Ran(T)+\qn(T)
\]
are closed. We also show that these three conditions are independent and that closedness alone is not sufficient on arbitrary Banach spaces.

Note that a positive answer to Mbekhta's conjecture was claimed in \cite{Amouch}, but the corresponding article was retracted; see \cite{Retraction}. For results on Kato-type decompositions and semi-B-Fredholm operators, see \cite{Labrousse1980,Muller2001,BerkaniSarih2001,AznayKato2023}. Whenever closed subspaces form a direct sum equal to the ambient Banach space, this decomposition is topological; it need not be orthogonal.
We allow zero summands and use the convention that an operator on the zero
space is nilpotent and quasinilpotent.
For $A\in\B(H)$, $A'$ denotes the Hilbert adjoint.

We use the following criterion from 
\cite[Lemma~2.2]{AznayChabbabi2026}: for any $A\in\B(X)$,
\[
 x\in \qn(A)\quad\Longleftrightarrow\quad
 \sup_{r\ge0}k^r\|A^rx\|<\infty\quad\text{for every integer }k\ge1.
\]
 The paper is organized as follows. In Section~2, we give a negative answer to Mbekhta's conjecture using a bilateral weighted shift. In Section~3, we establish a sufficient condition for pseudo-Fredholmness on Banach spaces in terms of $\cc(T)$, $\qn(T)$, $\Ran(T)$ and $\Ker(T)$, and obtain a characterization on Hilbert spaces. We also give some spectral counterexamples. In Section~4, we characterize pseudo-B-Fredholm and pseudo-B-Weyl operators. In Section~5, we show that pseudo-Fredholmness on general Banach spaces cannot be characterized solely in terms of the subspaces $\cc(T)$, $\qn(T)$, $\Ran(T^n)$ and $\Ker(T^n)$, $n\geq 1$, through the closedness or complementability of their sums and intersections. Finally, in Section~6, we prove that the three closedness conditions are independent and that none of the four assumptions in the Banach-space result can be omitted.

\section{On Mbekhta's conjecture and some related questions on pseudo-Fredholm operators}\label{sec:mbekhta_questions}
Let $H$ be a complex Hilbert space. In this section, we give a negative answer to Mbekhta's
conjecture on pseudo-Fredholm operators. We consider the bilateral weighted shift given
by Mbekhta in \cite[pp.~274--275]{Mbekhta}. This operator is not pseudo-Fredholm,
as proved in the same paper. We show that it satisfies the conditions of Mbekhta's
conjecture and, consequently, those of the conjecture stated in our earlier paper
\cite{AZO}.

Although some of the calculations below are not needed to obtain these counterexamples, they will be useful in the sequel.

Let $(e_j)_{j\in\mathbb Z}$ be the standard orthonormal basis of $\ell^2(\mathbb Z)$ and define
\begin{equation}\label{eq:bilshift}
 Be_j=a_je_{j+1},\qquad
 a_j=\begin{cases}1,&j\le0,\\2^{-j},&j\ge1.\end{cases}
\end{equation}
The weights satisfy $0<a_j\le1$, so for every $x=\sum_jx_je_j$,
\[
 \|Bx\|^2=\sum_j|a_jx_j|^2\le\sum_j|x_j|^2=\|x\|^2.
\]
Thus $B$ is a contraction. Since all weights are nonzero, $Bx=0$ forces every $x_j=0$, and
\begin{equation}\label{eq:Binjective}\Ker(B)=\{0\}.
\end{equation}
Define the closed coordinate subspaces
\[
 H_-=\overline{\spn}\{e_j:j\le0\},\qquad
 H_+=\overline{\spn}\{e_j:j\ge1\}.
\]
Then $\ell^2(\mathbb Z)=H_-\oplus H_+$ orthogonally.

\begin{proposition}\label{prop:bil}
The operator $B$ in \eqref{eq:bilshift} satisfies
\begin{equation}\label{eq:bilidentities}
 \Ran(B)+\qn(B)=\ell^2(\mathbb Z),\qquad
 \cc(B)+\qn(B)=\ell^2(\mathbb Z),\qquad
 \cc(B)\cap\Ker(B)=\{0\}.
\end{equation}
The operator is not pseudo-Fredholm by Mbekhta's original example \cite[pp.~274--275]{Mbekhta}.
\end{proposition}
\begin{proof}
We first show that $H_+\subseteq\qn(B)$.
For $j\ge1$, repeated application of $B$ gives
\[
 B^re_j=a_ja_{j+1}\cdots a_{j+r-1}e_{j+r}
 =2^{-\sum_{s=0}^{r-1}(j+s)}e_{j+r}
 =2^{-rj-r(r-1)/2}e_{j+r}.
\]
Because $j\ge1$,
\[
 |2^{-rj-r(r-1)/2}|\le2^{-r(r+1)/2}.
\]
For an arbitrary $x=\sum_{j\ge1}x_je_j\in H_+$, the vectors $B^re_j$ are orthogonal as $j$ varies. The Pythagorean theorem yields
\[
 \|B^rx\|^2=\sum_{j\ge1}|x_j|^2\,2^{-2rj-r(r-1)}
 \le2^{-r(r+1)}\|x\|^2.
\]
Consequently $\|B^rx\|^{1/r}\le2^{-(r+1)/2}\|x\|^{1/r}\to0$ for $x\ne0$; the assertion for $x=0$ is immediate. Hence
\begin{equation}\label{eq:posH0}H_+\subseteq\qn(B).\end{equation}

We next prove that $H_-\subseteq\cc(B)$.
For $x=\sum_{j\le0}x_je_j\in H_-$ define
\[
 x^{(r)}=\sum_{j\le0}x_je_{j-r}\qquad(r\ge0).
\]
The sums converge in $\ell^2$, and changing indices preserves the norm:
$\|x^{(r)}\|=\|x\|$. Since the index $j-r-1$ is nonpositive for every $j\le0$ and $r\ge0$, its weight is $1$. Thus
\[
 Bx^{(r+1)}=\sum_{j\le0}x_jBe_{j-r-1}
 =\sum_{j\le0}x_je_{j-r}=x^{(r)}.
\]
This is a backward chain with constant norm, so $x\in\cc(B)$. Therefore
\begin{equation}\label{eq:negK}H_-\subseteq\cc(B)\subseteq\Ran(B).\end{equation}

We now prove the first two equalities in \eqref{eq:bilidentities}.
For any $z\in\ell^2(\mathbb Z)$ write uniquely $z=z_-+z_+$ with $z_-\in H_-$ and $z_+\in H_+$. Equations \eqref{eq:posH0}--\eqref{eq:negK} imply
$z_-\in\Ran(B)\cap\cc(B)$ and $z_+\in\qn(B)$. Therefore $z$ belongs to both sums in \eqref{eq:bilidentities}. The third equality follows from \eqref{eq:Binjective}.

Finally, Mbekhta proved that this shift is not pseudo-Fredholm in
\cite[pp.~274--275]{Mbekhta}.
\end{proof}

It follows that the conjecture of Mbekhta \cite[p.~261]{Mbekhta} is false.
Moreover, $\dim(\cc(B)\cap\Ker(B))=0$ and
$\dim\big(\ell^2(\mathbb Z)/(\Ran(B)+\qn(B))\big)=0$.
Hence the sufficiency assertion of the conjecture in \cite{AZO} is also false.

Only one coefficient condition is needed for the quasinilpotent part of this shift.

\begin{proposition}\label{prop:bilcores}
For $x=\sum_{j\in\mathbb Z}x_je_j$, one has
\[
 x\in \qn(B)\quad\Longleftrightarrow\quad
 |x_{-N}|^{1/N}\longrightarrow0\qquad(N\to\infty).
\]
Moreover, both $\qn(B)$ and $\cc(B)$ are dense proper subspaces of $\ell^2(\mathbb Z)$ and hence are not closed.
\end{proposition}
\begin{proof}
If $x\in \qn(B)$, the first $N$ weights starting at index $-N$ are all equal to $1$. Orthogonality gives
$|x_{-N}|\le\|B^Nx\|$, so $|x_{-N}|^{1/N}\to0$.

Conversely, assume $|x_{-N}|^{1/N}\to0$. Put
$S_N=\bigl(\sum_{m\ge N}|x_{-m}|^2\bigr)^{1/2}$.
For any $0<\varepsilon<1$ we have $|x_{-m}|\le\varepsilon^m$ for all sufficiently large $m$, whence
$S_N\le\varepsilon^N(1-\varepsilon^2)^{-1/2}$ eventually. Thus $S_N^{1/N}\to0$. Write $x=x_-+x_+$ using $H_-\oplus H_+$, and for $r\ge2$ set $N=\lfloor r/2\rfloor$ and
\[
 x_-=u_r+v_r,\qquad
 u_r=\sum_{0\le m<N}x_{-m}e_{-m},\qquad
 v_r=\sum_{m\ge N}x_{-m}e_{-m}.
\]
Since $\|B\|\le1$, $\|B^rv_r\|\le S_N$. The orbit of each $e_{-m}$ in $u_r$ meets at least the positive weights
$2^{-1},\ldots,2^{-(r-N)}$, so orthogonality yields
\[
 \|B^ru_r\|\le2^{-(r-N)(r-N+1)/2}\|x\|.
\]
Both terms have $r$th roots tending to zero, since $N/r\to1/2$. Hence $x_-\in \qn(B)$. We already know $H_+\subseteq \qn(B)$, so $x\in \qn(B)$.

Every finitely supported vector satisfies the coefficient criterion, whereas $\sum_{m\ge1}m^{-1}e_{-m}$ does not; thus $\qn(B)$ is dense and proper. Every basis vector admits a bounded backward $B$-chain (only finitely many weights differ from $1$ on its backward orbit), so $\cc(B)$ is also dense. Finally $y=\sum_{m\ge1}2^{-m}e_m\notin \Ran(B)$: its preimage would have infinitely many coordinates equal to $1/2$. Since $\cc(B)\subseteq \Ran(B)$, the core is proper.
\end{proof}

\section{A sufficient condition on Banach spaces and the Hilbert-space characterization}\label{sec:characterization}
We first give a sufficient condition on Banach spaces in terms of the analytic core,
quasinilpotent part, kernel and range. The characterization on Hilbert spaces will follow as a corollary.

\begin{theorem}\label{th:suff}
Let $T\in\B(X)$. Assume that
\begin{align*}
\textup{(C1)}\quad &\cc(T)\text{ is closed};\\
\textup{(C2)}\quad &\cc(T)+\qn(T)\text{ is closed};\\
\textup{(C3)}\quad &\Ker(T)\cap\cc(T)\text{ is complemented in }X;\\
\textup{(C4)}\quad &\Ran(T)+\qn(T)\text{ is complemented in }X.
\end{align*}
Then $T$ is pseudo-Fredholm.
\end{theorem}

\begin{proof}
\par\medskip\noindent\textit{Step 1.} We begin with the restriction of $T$ to $K$ and the operators induced on the quotient spaces.\par\smallskip
Put
\[
K=\cc(T),\qquad L=\cc(T)+\qn(T),\qquad
M=\Ran(T)+\qn(T),\qquad A=\Ker(T)\cap K.
\]
Then $K\subseteq L\subseteq M$ are closed and $A,M$ are complemented in $X$. From \eqref{eq:corefacts},
\begin{equation}\label{eq:preimage}
TK=K,\qquad T^{-1}(L)=L.
\end{equation}

The restriction $S=T|_K$ is surjective with kernel $A$. Let $P_A$ be a bounded projection from $X$ onto $A$. Since $K=A\oplus(K\cap\Ker(P_A))$, the restriction of $S$ to $K\cap\Ker(P_A)$ is a bounded bijection onto $K$. Hence there exists $V\in\B(K)$ such that
\begin{equation}\label{eq:right}
SV=I_K,\qquad V(K)\subseteq K\cap\Ker(P_A).
\end{equation}
By \eqref{eq:right}, $(S-\lambda I)V=I-\lambda V$ and $S-\lambda I$ is surjective for $|\lambda|$ small.

The operator $Q$ induced by $T$ on $L/K$ is
quasinilpotent. Indeed, for every $z\in L/K$,
we may write $z=h+K$ with $h\in \qn(T)$. Hence
\[
\|Q^nz\|^{1/n}\leq\|T^nh\|^{1/n}
\longrightarrow 0.
\]
Thus $\qn(Q)=L/K$. For each integer $k\ge1$, the family $(k^nQ^n)_{n\ge0}$ is pointwise bounded. By the uniform boundedness principle, $\sup_{n\ge0}k^n\|Q^n\|<\infty$. Therefore $r(Q)\le1/k$ for every $k\ge1$, and $Q$ is quasinilpotent.
In particular, $Q-\lambda I$ is invertible for $\lambda\ne0$.

Let $D$ be the operator induced by $T$ on the Banach space $Y=X/L$. By \eqref{eq:preimage}, $D$ is injective and
\[
\Ran(D)=M/L.
\]

Since $M$ is closed, $D$ is bounded below.
Since $M$ is complemented in $X$ and contains
$L$, $M/L$ is complemented in $Y$.
Fix a closed complement $F$ of $M$ in $X$ and put
\[
\widehat F=(F+L)/L.
\]
Since $L\subseteq M$ and $X=M\oplus F$, the
subspace $\widehat F$ is closed in $Y$, and
\[
Y=(M/L)\oplus\widehat F.
\]

\par\medskip\noindent\textit{Step 2.} We describe the range of $D-\lambda I_Y$ for small $\lambda$.\par\smallskip
Let $\pi:Y\to M/L$ be the bounded projection
onto $M/L$ along $\widehat F$. Since $\pi D=D$,
we have
\[
D_\lambda:=\pi(D-\lambda I_Y)=D-\lambda\pi.
\]
As $D:Y\to M/L$ is an isomorphism, we can write
\[
D_\lambda=D(I_Y-\lambda D^{-1}\pi).
\]
By the Neumann series,
\[
D_\lambda^{-1}
=\sum_{n=0}^{\infty}
\lambda^n(D^{-1}\pi)^nD^{-1}
\]
for sufficiently small $|\lambda|$.
Hence $D_\lambda:Y\to M/L$ is an isomorphism
and its inverse depends holomorphically
on $\lambda$.

Let $\pi_{\widehat F}:Y\to\widehat F$ be the
projection onto $\widehat F$ along $M/L$.
Since $\pi_{\widehat F}D=0$, we obtain
\[
(D-\lambda I_Y)D_\lambda^{-1}u
=u-\lambda\pi_{\widehat F}D_\lambda^{-1}u,
\qquad u\in M/L.
\]
Moreover, for every $y\in Y$,
\[
\|y\|
\leq \|D_\lambda^{-1}\|\|D_\lambda y\|
\leq \|D_\lambda^{-1}\|\|\pi\|
\|(D-\lambda I_Y)y\|.
\]
Thus $D-\lambda I_Y$ is bounded below, and
\[
\Ran(D-\lambda I_Y)
=\left\{
u-\lambda\pi_{\widehat F}D_\lambda^{-1}u:
u\in M/L
\right\}.
\]
Hence its range is the graph of the bounded
operator
\[
\Phi(\lambda):=-\lambda\pi_{\widehat F}D_\lambda^{-1}
:M/L\longrightarrow\widehat F,
\]
which depends holomorphically on $\lambda$
and vanishes at $\lambda=0$.

\par\medskip\noindent\textit{Step 3.} We find complements of the kernel and range of $T-\lambda I$ which do not depend on $\lambda$.\par\smallskip
We next obtain fixed complements of the kernel and range
of $T-\lambda I$. Set $E=\Ker(P_A)$, so that
$X=A\oplus E$. For $0<|\lambda|$ sufficiently
small, the injectivity of $D-\lambda I_Y$ implies
that $\Ker(T-\lambda I)\subseteq L$. Since
$Q-\lambda I$ is invertible on $L/K$, we obtain
$\Ker(T-\lambda I)\subseteq K$. Hence
\[
\Ker(T-\lambda I)
=\Ker(S-\lambda I)
=(I_K-\lambda V)^{-1}A.
\]
Since $V(K)\subseteq E$, for every $a\in A$,
\[
(I_K-\lambda V)^{-1}a
=a+\lambda V(I_K-\lambda V)^{-1}a.
\]
Thus $\Ker(T-\lambda I)$ is the graph, relative
to $X=A\oplus E$, of the bounded operator
\[
\Theta(\lambda)
=\lambda V(I_K-\lambda V)^{-1}|_A
:A\longrightarrow E.
\]
Consequently,
\begin{equation}\label{eq:kercompl}
X=\Ker(T-\lambda I)\oplus E.
\end{equation}
Moreover, the projection onto $\Ker(T-\lambda I)$
along $E$ is given by
\[
P_{\mathrm{ker}}(\lambda)=(I_K-\lambda V)^{-1}P_A.
\]
By the Neumann series, $P_{\mathrm{ker}}(\lambda)$ extends
holomorphically to $\lambda=0$, with
$P_{\mathrm{ker}}(0)=P_A$.

The surjectivity of $S-\lambda I$ and the
invertibility of $Q-\lambda I$ imply
\[
L\subseteq\Ran(T-\lambda I).
\]
Indeed, for every $x\in L$, there exists
$y\in L$ such that
\[
x-(T-\lambda I)y\in K.
\]
Since $S-\lambda I$ is surjective on $K$,
there exists $k\in K$ such that
\[
x-(T-\lambda I)y=(T-\lambda I)k.
\]
Hence $x=(T-\lambda I)(y+k)$, as required.

Since $L\subseteq\Ran(T-\lambda I)$, we have
\[
\Ran(T-\lambda I)
=q^{-1}(\Ran(D-\lambda I_Y)),
\]
where $q:X\to X/L$ is the quotient map.
Let $J:F\to\widehat F$ be defined by
$Jf=f+L$. Since $X=M\oplus F$ and
$L\subseteq M$, the map $J$ is a bounded
isomorphism.

Recall that $\Ran(D-\lambda I_Y)$ is
the graph of the holomorphic operator
$\Phi(\lambda):M/L\to\widehat F$.
Put
\[
\Psi(\lambda)
=J^{-1}\Phi(\lambda)q|_M
:M\longrightarrow F.
\]
Then
\[
\Ran(T-\lambda I)
=\{m+\Psi(\lambda)m:m\in M\}.
\]
Thus $\Ran(T-\lambda I)$ is the graph,
relative to $X=M\oplus F$, of the bounded
operator $\Psi(\lambda)$, which depends
holomorphically on $\lambda$ and satisfies
$\Psi(0)=0$. Consequently,
\begin{equation}\label{eq:rancompl}
X=\Ran(T-\lambda I)\oplus F.
\end{equation}
Let $P_M$ be the projection onto $M$ along $F$.
The projection onto $\Ran(T-\lambda I)$ along
$F$ is given by
\[
P_{\mathrm{ran}}(\lambda)=(I_M+\Psi(\lambda))P_M=P_M+\Psi(\lambda)P_M.
\]
Since $\Psi(\lambda)$ is holomorphic near zero
and $\Psi(0)=0$, the projection $P_{\mathrm{ran}}(\lambda)$
extends holomorphically to $\lambda=0$, with
$P_{\mathrm{ran}}(0)=P_M$. All assertions hold for
$0<|\lambda|<\varepsilon$, for some
$\varepsilon>0$.

\par\medskip\noindent\textit{Step 4.} We construct a generalized inverse of $T-\lambda I$.\par\smallskip
By \eqref{eq:kercompl}, the restriction
\[
(T-\lambda I)|_E:E\longrightarrow
\Ran(T-\lambda I)
\]
is an isomorphism. Since $\Ran(T-\lambda I)$
is the graph of $\Psi(\lambda):M\to F$,
the restriction of $P_M$ to $\Ran(T-\lambda I)$
is also an isomorphism onto $M$. Hence
\[
B(\lambda)=P_M(T-\lambda I)|_E:
E\longrightarrow M
\]
is a bounded isomorphism. Define
\[
G(\lambda)=B(\lambda)^{-1}P_M
\in\B(X),\qquad 0<|\lambda|<\varepsilon.
\]
Since $B(\lambda)$ is holomorphic and invertible for $0<|\lambda|<\varepsilon$, so is $G(\lambda)$.
Since
\[
B(\mu)-B(\lambda)
=(\lambda-\mu)P_M|_E,
\]
we obtain 
\begin{equation}\label{eq:resid}
\begin{aligned}
G(\lambda)-G(\mu)
&=B(\lambda)^{-1}
(B(\mu)-B(\lambda))
B(\mu)^{-1}P_M\\
&=(\lambda-\mu)G(\lambda)G(\mu).
\end{aligned}
\end{equation}

Moreover, since $G(\lambda)=B(\lambda)^{-1}P_M$
and $\Ran(T-\lambda I)$ is the graph of
$\Psi(\lambda)$, we have
\[
(T-\lambda I)G(\lambda)
=P_M+\Psi(\lambda)P_M=P_{\mathrm{ran}}(\lambda).
\]
Thus $(T-\lambda I)G(\lambda)$ is the
projection onto $\Ran(T-\lambda I)$ along $F$.
Similarly, since $B(\lambda)^{-1}P_M(T-\lambda I)$
is the projection onto $E$ along
$\Ker(T-\lambda I)$, we obtain
\[
I-G(\lambda)(T-\lambda I)
=(I_K-\lambda V)^{-1}P_A.
\]
Hence both projections extend holomorphically
to $\lambda=0$, with respective values
$P_M$ and $P_A$.

\par\medskip\noindent\textit{Step 5.} We use $G(\lambda)$ to define a projection $P$.\par\smallskip
Choose $0<\rho<\varepsilon$ and set
\begin{equation}\label{eq:Riesz}
P=-\frac{1}{2\pi i}
\int_{|\lambda|=\rho}G(\lambda)\,d\lambda.
\end{equation}
We show that $P$ is a bounded projection commuting
with $T$. Fix $0<|\lambda_0|<\varepsilon$ and put
$C=G(\lambda_0)$. By \eqref{eq:resid},
\[
G(\lambda)-C
=(\lambda-\lambda_0)G(\lambda)C.
\]
Interchanging $\lambda$ and $\lambda_0$ gives
$G(\lambda)C=CG(\lambda)$. Consequently,
\[
[I-(\lambda-\lambda_0)C]^{-1}
=I+(\lambda-\lambda_0)G(\lambda),
\]
and hence
\[
G(\lambda)
=C[I-(\lambda-\lambda_0)C]^{-1}.
\]

For $\lambda\ne\lambda_0$, the operator
$I-(\lambda-\lambda_0)C$ is invertible
if and only if
\[
(\lambda-\lambda_0)^{-1}\notin\sigma(C).
\]
For $0<|\lambda|<\varepsilon$, this operator is
invertible. Since $\lambda\mapsto(\lambda-\lambda_0)^{-1}$
is biholomorphic near zero, a punctured neighborhood
of $c_0=-1/\lambda_0$ is disjoint from $\sigma(C)$.
Thus, if $c_0\in\sigma(C)$, it is isolated. Notice that
$c_0\ne0$ since $\lambda_0\ne0$.

Let $P_0$ be the Riesz projection of $C$
associated with $c_0$, and put $P_0=0$
if $c_0\notin\sigma(C)$. Since $G(\lambda)$
commutes with $C$,  $P_0$ commutes
with $G(\lambda)$.

Put $Z=(I-P_0)X$ and $C_1=C|_Z$.
Since $c_0\notin\sigma(C_1)$ and
$c_0=-1/\lambda_0$, the operator
$I+\lambda_0C_1$ is invertible. Moreover,
\[
\begin{aligned}
I-(\lambda-\lambda_0)C_1
&=I+\lambda_0C_1-\lambda C_1\\
&=(I+\lambda_0C_1)
\left[I-\lambda(I+\lambda_0C_1)^{-1}C_1\right].
\end{aligned}
\]
By the Neumann series, the operator on the
left is invertible for sufficiently small
$|\lambda|$, and its inverse depends
holomorphically on $\lambda$. Since
\[
G(\lambda)|_Z
=C_1[I-(\lambda-\lambda_0)C_1]^{-1},
\]
it follows that $(I-P_0)G(\lambda)$ extends
holomorphically to $\lambda=0$.
Consequently, by Cauchy's theorem,
\[
\int_{|\lambda|=\rho}
(I-P_0)G(\lambda)\,d\lambda=0.
\]

Suppose now that $P_0\ne0$ and put
$C_0=C|_{P_0X}$. Since
$\sigma(C_0)=\{c_0\}$ and $c_0\ne0$,
$C_0$ is invertible. Moreover,
\[
G(\lambda)|_{P_0X}
=(C_0^{-1}+\lambda_0 I-\lambda I)^{-1}.
\]
The operator
\[
A_0=C_0^{-1}+\lambda_0 I
\]
is quasinilpotent, since
\[
\sigma(A_0)
=\{c_0^{-1}+\lambda_0\}
=\{0\}.
\]
Thus, by the Riesz integral formula,
\[
-\frac{1}{2\pi i}
\int_{|\lambda|=\rho}
G(\lambda)|_{P_0X}\,d\lambda
=I_{P_0X}.
\]
Together with the preceding integral,
this shows that $P=P_0$. The same
conclusion holds when $P_0=0$, since
$G(\lambda)$ then extends holomorphically
to zero. Consequently,
\begin{equation}\label{eq:PG}
P^2=P,\qquad PG(\lambda)=G(\lambda)P,
\end{equation}
and $(I-P)G(\lambda)$ extends
holomorphically to $\lambda=0$.

\par\medskip\noindent\textit{Step 6.} We show that $P$ commutes with $T$.\par\smallskip
We next show that $P$ commutes with $T$.
Since
\[
TG(\lambda)-G(\lambda)T
=(T-\lambda I)G(\lambda)
-G(\lambda)(T-\lambda I),
\]
and both terms on the right extend
holomorphically to zero, so does their
difference. Hence, by Cauchy's theorem,
\[
\begin{aligned}
TP-PT
&=-\frac{1}{2\pi i}
\int_{|\lambda|=\rho}
\bigl(TG(\lambda)-G(\lambda)T\bigr)
\,d\lambda\\
&=0.
\end{aligned}
\]
Thus $P$ is a bounded projection
commuting with $T$.

\par\medskip\noindent\textit{Step 7.} We prove that $T|_{P(X)}$ is quasinilpotent.\par\smallskip
We now show that $T|_{P(X)}$ is
quasinilpotent. Put
\[
H(\lambda)=(T-\lambda I)G(\lambda).
\]
As proved above, $H(\lambda)$ extends
holomorphically to $\lambda=0$.
For every $n\ge1$, we have
\[
T^nG(\lambda)-\lambda^nG(\lambda)
=
\sum_{k=0}^{n-1}
\lambda^k T^{n-1-k}H(\lambda).
\]
Therefore,
$T^nG(\lambda)-\lambda^nG(\lambda)$
extends holomorphically to zero.
By Cauchy's theorem and
\eqref{eq:Riesz},
\[
T^nP
=-\frac{1}{2\pi i}
\int_{|\lambda|=\rho}
\lambda^nG(\lambda)\,d\lambda.
\]
Hence
\[
\|T^nP\|
\leq
\rho^{n+1}
\max_{|\lambda|=\rho}\|G(\lambda)\|.
\]
The integral defining $P$ is independent
of $\rho$, by Cauchy's theorem on the
punctured disk. Consequently,
\[
r(T|_{P(X)})\leq\rho
\]
for every sufficiently small $\rho>0$.
Thus $T|_{P(X)}$ is quasinilpotent.

\par\medskip\noindent\textit{Step 8.} We prove that $T|_{(I-P)X}$ is semi-regular.\par\smallskip
Finally, let $Z=(I-P)X$ and
$U=T|_Z$. Since $P$ commutes with
$T$ and $G(\lambda)$, the restriction
\[
g(\lambda)=G(\lambda)|_Z
\]
is well defined and extends
holomorphically to $\lambda=0$.

Since $(T-\lambda I)G(\lambda)$ is the projection
onto $\Ran(T-\lambda I)$ along $F$, and
$G(\lambda)(T-\lambda I)$ is the projection
onto $E$ along $\Ker(T-\lambda I)$, we have
\[
(T-\lambda I)G(\lambda)(T-\lambda I)
=T-\lambda I
\]
and
\[
G(\lambda)(T-\lambda I)G(\lambda)
=G(\lambda).
\]
Thus $G(\lambda)$ is a generalized inverse
of $T-\lambda I$.

As $P$ commutes with both $T$ and $G(\lambda)$,
restricting the first identity to $Z=(I-P)X$
gives
\[
(U-\lambda I)g(\lambda)(U-\lambda I)
=U-\lambda I.
\]
Passing to the limit as $\lambda\to0$, we obtain
\[
Ug(0)U=U.
\]
Hence $Ug(0)$ is a bounded projection onto
$\Ran(U)$, so $\Ran(U)$ is closed.

Let $x\in\Ker(U)$. The function
\[
f(\lambda)
=[I-g(\lambda)(U-\lambda I)]x
\]
is holomorphic near zero and satisfies
\[
(U-\lambda I)f(\lambda)=0,
\qquad f(0)=x.
\]
Writing
\[
f(\lambda)=\sum_{n=0}^{\infty}
x_n\lambda^n,
\]
we obtain
\[
x_0=x,\qquad Ux_{n+1}=x_n
\quad(n\ge0).
\]
Hence  $x\in\Ran^\infty(U)$. Thus
$\Ker(U)\subseteq\Ran^\infty(U)$.
Since $\Ran(U)$ is closed, $U$ is
semi-regular.

We have proved that
\[
X=(I-P)X\oplus P(X)
\]
is a generalized Kato decomposition of
$T$. Consequently, $T$ is
pseudo-Fredholm.

\end{proof}

The Hilbert-space characterization now follows directly.

\begin{corollary}\label{thm:characterization}
Let $H$ be a complex Hilbert space and let $T\in\B(H)$. The following statements are equivalent:
\begin{enumerate}[label=(\roman*)]
\item $T$ is pseudo-Fredholm.
\item $\cc(T)$, $\cc(T)+\qn(T)$ and $\Ran(T)+\qn(T)$ are closed in $H$.
\end{enumerate}
\end{corollary}
\begin{proof}
Suppose (ii). Since every closed subspace of $H$ is complemented,
$\Ker(T)\cap \cc(T)$ and $\Ran(T)+\qn(T)$ are complemented in $H$.
Theorem~\ref{th:suff} gives (i).

Conversely, suppose (i).
Let $T=S\oplus Q$ be a generalized Kato decomposition, where $S$ is semi-regular and $Q$ is quasinilpotent. By \cite[Proposition~1.4]{Mbekhta}, $\cc(S)=\Ran^\infty(S)$ is closed, $\qn(S)\subseteq \cc(S)\subseteq \Ran(S)$, $\cc(Q)=0$ and $\qn(Q)=N$. Hence
\[
\left\{
 \begin{aligned}
 \cc(T)&=\cc(S)\oplus\{0\},\\
 \cc(T)+\qn(T)&=\cc(S)\oplus N,\\
 \Ran(T)+\qn(T)&=\Ran(S)\oplus N.
 \end{aligned}
\right.
\]
Thus all three spaces are closed.

\end{proof}

\begin{remark}
The necessity of the three closedness conditions holds on Banach spaces as well, by the same direct-sum argument. Their sufficiency on Hilbert spaces follows from the automatic complementability of closed subspaces. Section~\ref{sec:tsunified} shows why an additional assumption is needed on Banach spaces.
\end{remark}

In the following example, $\sigma_{\mathrm{ap}}(T)$ and
$\sigma_{\mathrm{su}}(T)$ denote the approximate point spectrum
and the surjective spectrum of $T$, respectively. We write
$\operatorname{acc}(A)$ for the set of accumulation points of
$A\subseteq\mathbb C$, and
$\mathbb T=\{z\in\mathbb C:|z|=1\}$.

Thus $\sigma_{\mathrm{ap}}(T)=\{\lambda\in\mathbb C:T-\lambda I\text{ is not bounded below}\}$ and $\sigma_{\mathrm{su}}(T)=\{\lambda\in\mathbb C:T-\lambda I\text{ is not surjective}\}$.

The following example shows that
$0\notin\operatorname{acc}\sigma_{\mathrm{ap}}(T)$
does not imply that $T$ is pseudo-Fredholm. It will also
be used in Section~\ref{sec:independence}.

\begin{example}\label{rem:approx-spectrum-counterexample}
Let $E=\ell^2(\mathbb N)$, $H=\ell^2(\mathbb N_0,E)$,
and let $U$ be the unilateral shift on $H$. Define
\[
Be_n=\frac1n e_n,\qquad Cy=(By,0,\ldots),
\]
and consider
\[
T=\begin{pmatrix}U&C\\0&0\end{pmatrix}
\in\B(H\oplus E).
\]
For every $m\ge1$ and $(x,y)\in H\oplus E$, we have
\[
T^m(x,y)=\bigl(U^{m-1}(Ux+Cy),0\bigr).
\]
Since $U$ is an isometry and the ranges of $U$ and $C$
are orthogonal,
\[
\|T^m(x,y)\|^2=\|x\|^2+\|By\|^2.
\]
As $B$ is injective, it follows that  $\qn(T)=\{0\}$. Moreover,
\[
\|T(0,e_n)\|=\frac1n\longrightarrow0.
\]
Thus $\Ran(T)$ is not closed, and $T$ is not semi-regular.
Since any quasinilpotent summand in a generalized Kato
decomposition of $T$ must be contained in $\qn(T)$,
we conclude that $T$ is not pseudo-Fredholm.

Now, $U-\lambda I$ is bounded below for $|\lambda|<1$
and invertible for $|\lambda|>1$. Hence
\[
T-\lambda I=
\begin{pmatrix}
U-\lambda I&C\\
0&-\lambda I
\end{pmatrix}
\]
is bounded below whenever $\lambda\ne0$ and
$|\lambda|\ne1$. On the other hand,
$0\in\sigma_{\mathrm{ap}}(T)$ and
\[
\mathbb T=\sigma_{\mathrm{ap}}(U)
\subseteq\sigma_{\mathrm{ap}}(T).
\]
Consequently,
\[
\sigma_{\mathrm{ap}}(T)=\{0\}\cup\mathbb T,
\qquad
0\notin\operatorname{acc}\sigma_{\mathrm{ap}}(T).
\]

This gives an affirmative answer to the approximate
point spectrum part of
\cite{CvetkovicZivkovic2017} (Question~4.4 in the arXiv version).
Taking the Hilbert adjoint gives the corresponding
example for the surjective spectrum, since
\[
\sigma_{\mathrm{su}}(T')
=\overline{\sigma_{\mathrm{ap}}(T)}
\]
and the existence of a generalized Kato decomposition
is preserved under taking Hilbert adjoints.
\end{example}

\section{Characterizations of pseudo-B-Fredholm and
pseudo-B-Weyl operators in terms of canonical subspaces}

The following theorem gives a corrected version of our
conjecture stated in \cite[p.~202]{AZO}.

\begin{theorem}\label{th:pseudo-B-Fredholm}
Let $X$ be a complex Banach space and
$T\in\mathcal B(X)$. Then $T$ is pseudo-B-Fredholm
if and only if the following conditions hold:
\begin{enumerate}[label=(\roman*)]
\item $\cc(T)$ is closed;
\item $\cc(T)+\qn(T)$ is closed;
\item $\dim(\cc(T)\cap \Ker(T))<\infty$ and 
$\mbox{codim}\,(\Ran(T)+\qn(T))<\infty.$
\end{enumerate}
\end{theorem}

\begin{proof}
Suppose first that $T$ is pseudo-B-Fredholm.
Since every Fredholm operator admits a Kato-type
decomposition, we may write $T=S\oplus Q$,
where $S$ is semi-regular and Fredholm, and
$Q$ is quasinilpotent.

Since $\cc(S)=\Ran^\infty(S)$ is closed and
$\qn(S)\subseteq \cc(S)\subseteq \Ran(S)$, we have
\[
\begin{aligned}
\cc(T)&=\cc(S)\oplus\{0\},\\
\cc(T)+\qn(T)&=\cc(S)\oplus N,\\
\cc(T)\cap \Ker(T)&=\Ker(S)\oplus\{0\},\\
\Ran(T)+\qn(T)&=\Ran(S)\oplus N,
\end{aligned}
\]
where $N$ denotes the space on which $Q$ acts.
Since $S$ is Fredholm, (i)--(iii) follow.

Conversely, suppose that (i)--(iii) hold.
Put
\[
L=\cc(T)+\qn(T),\qquad
M=\Ran(T)+\qn(T).
\]
Since $L$ is closed and $T^{-1}(L)=L$,
the operator induced by $T$ on $X/L$
is bounded and has range $M/L$.
By (iii), this range has finite codimension. Since it is the range of a bounded operator between Banach spaces, the open mapping theorem, applied after adjoining a finite-dimensional algebraic complement, shows that it is closed. Hence $M$ is closed and complemented in $X$.

Moreover, by (iii), $\Ker(T)\cap \cc(T)$
is finite-dimensional and therefore
complemented in $X$.
Thus Theorem~\ref{th:suff} implies that
$T$ is pseudo-Fredholm.

Let $T=S\oplus Q$ be a generalized Kato
decomposition on $X=Y\oplus N$, where
$S\in\mathcal B(Y)$ is semi-regular and
$Q\in\mathcal B(N)$ is quasinilpotent.
Then
\[
\dim \Ker(S)=\dim(\cc(T)\cap \Ker(T))<\infty
\]
and
\[
\dim\frac{Y}{\Ran(S)}
=\dim\frac{X}{\Ran(T)+\qn(T)}<\infty.
\]
Since $\Ran(S)$ is closed, $S$ is Fredholm.
Consequently, $T$ is pseudo-B-Fredholm.
\end{proof}

As a consequence, we obtain the following
characterization of pseudo-B-Weyl operators.

\begin{corollary}\label{cor:pseudo-B-Weyl}
Let $X$ be a complex Banach space and
$T\in\mathcal B(X)$. Then $T$ is pseudo-B-Weyl
if and only if the following conditions hold:
\begin{enumerate}[label=(\roman*)]
\item $\cc(T)$ is closed;
\item $\cc(T)+\qn(T)$ is closed;
\item
\[
\dim(\cc(T)\cap \Ker(T))
=
\mbox{codim}\,{\Ran(T)+\qn(T)}
<\infty.
\]
\end{enumerate}
\end{corollary}

\begin{proof}
Recall that $T$ is pseudo-B-Weyl if and only
if it is pseudo-B-Fredholm with index \[
\operatorname{ind}(T)
:=
\dim(\cc(T)\cap \Ker(T))
-
\mbox{codim}\,{(\Ran(T)+\qn(T))}=0, \text{ see } \cite{AZO}\]
By Theorem~\ref{th:pseudo-B-Fredholm},
$T$ is pseudo-B-Fredholm if and only if
(i), (ii) and the finiteness of both
dimensions in (iii) hold.

Hence $\operatorname{ind}(T)=0$ if and only
if (iii) holds.
\end{proof}

\section{The obstruction on Banach spaces}\label{sec:tsunified}
We give an example showing that the closedness conditions in Corollary~\ref{thm:characterization} are not sufficient on arbitrary Banach spaces. It also shows that the kernel complementability assumption in Theorem~\ref{th:suff} cannot be omitted.

We first establish a proposition which will be used in the example. For $n\ge0$,
$T_{[n]}$ denotes the restriction of $T$ to $\Ran(T^n)$.  Recall that $T$ is upper (respectively, lower) semi-B-Fredholm
if $\Ran(T^n)$ is closed and $T_{[n]}$ is upper (respectively, lower) semi-Fredholm for some such
$n$; it is B-Fredholm if both conditions hold.
We use the degree of stable iteration
\[
 \operatorname{dis}(T)=\inf\{d\ge0:\Ker(T)\cap \Ran(T^n)=\Ker(T)\cap \Ran(T^d)
                    \text{ for all }n\ge d\}.
\]
An operator of finite degree $d$ is quasi-Fredholm if
$\Ran(T)+\Ker(T^d)$ and $\Ker(T)\cap \Ran(T^d)$ are closed;
see \cite{Labrousse1980,Muller2001}. The norm on $E\oplus_1 E$ is
$\|(x,y)\|=\|x\|+\|y\|$.

\begin{proposition}\label{thm:tsmaster}
Let $E\ne\{0\}$ be a complex Banach space, let $S\in\mathcal B(E)$ be surjective, and define on $X=E\oplus_1E$
\[
 T=T_S=\begin{pmatrix}S&I_E\\0&0\end{pmatrix},\qquad
 T(x,y)=(Sx+y,0).
\]
The following assertions hold.
\begin{enumerate}[label=\textup{(\roman*)}]
\item For every $n\ge1$,
\[
 T^n(x,y)=(S^{n-1}(Sx+y),0),\qquad
 \Ran(T^n)=E\oplus\{0\}.
\]
Moreover,
\[
 \cc(T)=E\oplus\{0\},\qquad
 \qn(T)=\{(x,y):Sx+y\in \qn(S)\}.
\]
In particular,
\[
 \cc(T)+\qn(T)=\Ran(T)+\qn(T)=X,\qquad
 \cc(T)\cap \Ker(T)=\Ker(S)\oplus\{0\}.
\]
\item The operator $T$ is quasi-Fredholm of degree $\operatorname{dis}(T)=1$ and lower semi-B-Fredholm.
\item The following are equivalent:
\begin{enumerate}[label=\textup{(\alph*)}]
\item $T$ is pseudo-Fredholm;
\item $T$ admits a Kato-type decomposition;
\item $T$ admits a direct-sum decomposition into a surjective operator and the zero operator;
\item $S$ has a bounded right inverse;
\item $\Ker(S)$ is complemented in $E$.
\end{enumerate}
\end{enumerate}
\end{proposition}
\begin{proof}
The power formula is immediate. As $S$ is surjective, all the ranges $\Ran(T^n)$ equal $E\oplus\{0\}$, while
\[
 \Ker(T)=\{(x,-Sx):x\in E\}.
\]
Since $S$ is surjective, the open mapping theorem permits the successive choice of preimages $Sx_{r+1}=x_r$ with $\|x_r\|\le c^r\|x_0\|$ for a fixed $c\ge1$. This proves $\cc(S)=E$. Together with $\cc(T)\subseteq \Ran^\infty(T)$, it gives $\cc(T)=E\oplus\{0\}$. The formula for $\qn(T)$ follows from
$\|T^n(x,y)\|=\|S^{n-1}(Sx+y)\|$ for $n\ge1$.

For $(x,y)\in X$, choose $u\in E$ with $Su=-y$; then
$(x,y)=(x-u,0)+(u,y)$ with $(u,y)\in \Ker(T)\subseteq \qn(T)$.
This proves the sum identities and the formula for the intersection. Since
$\Ker(T)\cap \Ran(T^n)=\Ker(S)\oplus\{0\}$ for every $n\ge1$, but $\Ker(T)$ is strictly larger, $\operatorname{dis}(T)=1$. Both $\Ran(T)+\Ker(T)$ and $\Ker(T)\cap \Ran(T)$ are closed, so $T$ is quasi-Fredholm. On $\Ran(T^n)=E\oplus\{0\}$, $T$ restricts to $S$, which is surjective. Hence $T$ is lower semi-B-Fredholm.

If $SR=I_E$ with $R$ bounded, put
\[
 M=E\oplus\{0\},\qquad N=\{(-Ry,y):y\in E\}.
\]
They are closed invariant complements, $T|_M$ is surjective, and $T|_N=0$. Thus (d) implies (c), (b) and (a).

Conversely, suppose $X=M\oplus N$ is a generalized Kato decomposition of $T$, with $T|_M$ semi-regular and $T|_N$ quasinilpotent. As $\cc(T|_N)=\{0\}$, the direct-sum formula for analytic cores gives
$E\oplus\{0\}=\cc(T)\subseteq M$. If $(x,y)\in M$ with $y\ne0$, choose $u\in E$ such that $Su=-y$. Since $E\oplus\{0\}\subseteq M$, we have $(u,y)\in M\cap \Ker(T)$, but $(u,y)\notin \Ran(T|_M)\subseteq E\oplus\{0\}$, contradicting semi-regularity. Hence $M=E\oplus\{0\}$.
The closed complement $N$ is the graph of a bounded operator $U:E\to E$. Invariance and $T(N)\subseteq M$ give $T(N)=\{0\}$, so $SU=-I_E$. Hence $-U$ is a bounded right inverse of $S$, proving (a)$\Rightarrow$(d). Since the decomposition in (b) is a generalized Kato decomposition, (b)$\Rightarrow$(a) also holds; (c)$\Rightarrow$(a) is immediate.
Finally, $SR=I_E$ gives the bounded projection $I-RS$ onto $\Ker(S)$. Conversely, a closed complement of $\Ker(S)$ yields a bounded inverse to the restriction of $S$ to that complement. Thus (d)$\Leftrightarrow$(e).
\end{proof}

\begin{corollary}\label{cor:banachobstruction}
Let $S$ be a surjective operator on a complex Banach space $E$ whose kernel is not complemented, and let $T_S$ be as in Proposition~\ref{thm:tsmaster}. Then $T_S$ is quasi-Fredholm of degree one and lower semi-B-Fredholm. Moreover,
\[
 \cc(T_S)=\Ran(T_S)=E\oplus\{0\},\qquad \cc(T_S)+\qn(T_S)=\Ran(T_S)+\qn(T_S)=X.
\]
In particular, the three subspaces $\cc(T_S)$, $\cc(T_S)+\qn(T_S)$ and $\Ran(T_S)+\qn(T_S)$ are closed, but $T_S$ is not pseudo-Fredholm. It does not admit a Kato-type decomposition or a decomposition into a semi-Fredholm operator and a nilpotent operator.
\end{corollary}
\begin{proof}
Since $\Ker(S)$ is not complemented, $S$ has no bounded linear right inverse. The assertions follow from Proposition~\ref{thm:tsmaster}.
\end{proof}

\begin{example}\label{ex:separable}
We give a concrete separable Banach-space example. Let
$q:\ell^1\to\ell^2$ be a bounded surjection; its existence follows from the classical fact that every separable Banach space is a quotient of $\ell^1$ (see \cite{Carothers}).

Put $E=\ell^1\oplus_1\ell^2$. For $a=(a_j)\in\ell^1$, write
$a_{\rm o}=(a_{2j-1})$ and $a_{\rm e}=(a_{2j})$ and define
\[
 S(a,b)=(a_{\rm o},q(a_{\rm e})),\qquad (a,b)\in E.
\]
We first check boundedness and surjectivity. For $(a,b)\in E$,
\[
 \|S(a,b)\|=\|a_{\rm o}\|_1+\|q(a_{\rm e})\|_2
 \le\max\{1,\|q\|\}\|a\|_1
 \le\max\{1,\|q\|\}\|(a,b)\|.
\]
Given $(u,v)\in\ell^1\oplus_1\ell^2$, choose $d\in\ell^1$ with
$q(d)=v$. Interleave $u$ and $d$ to obtain $a$ with
$a_{\rm o}=u$ and $a_{\rm e}=d$. Then $S(a,0)=(u,v)$.

Suppose that $S$ has a bounded right inverse $R$. For $z\in\ell^2$,
write $R(0,z)=(a(z),b(z))$ and put $Lz=a(z)_{\rm e}$.
The map $L:\ell^2\to\ell^1$ is bounded and linear. The equality
$SR(0,z)=(0,z)$ implies $qLz=z$, so $qL=I_{\ell^2}$.
Let $(e_n)$ be the canonical basis of $\ell^2$. Since $e_n\to0$
weakly, $Le_n\to0$ weakly in $\ell^1$. The Schur property gives
$\|Le_n\|_1\to0$, whereas
\[
 1=\|e_n\|_2=\|qLe_n\|_2\le\|q\|\,\|Le_n\|_1,
\]
a contradiction. Thus $S$ has no bounded right inverse, and $\Ker(S)$
is not complemented. By Corollary~\ref{cor:banachobstruction}, the operator $T_S$ on $E\oplus_1 E$ satisfies all three closedness conditions but is not pseudo-Fredholm. Moreover, $T_S$ is not pseudo semi-B-Fredholm. Indeed, otherwise its semi-Fredholm summand would admit a Kato decomposition, and $T_S$ would be pseudo-Fredholm, a contradiction. In particular, $T_S$ is not pseudo-B-Fredholm. However, it is lower semi-B-Fredholm by Proposition~\ref{thm:tsmaster}.
\end{example}

\begin{remark}\label{rem:history}
\begin{enumerate}    
\item In \cite[Theorem~3.2.1]{Labrousse1980}, Labrousse proved that a quasi-Fredholm operator on a Hilbert space admits a Kato-type decomposition. In the Banach-space case, M\"uller \cite[Theorem~5]{Muller2001} obtained the same conclusion by assuming, in addition, that $\Ran(T)+\Ker(T^d)$ and $\Ker(T)\cap \Ran(T^d)$ are complemented, where $d=\operatorname{dis}(T)$. These assumptions are automatically satisfied on Hilbert spaces. The preceding example shows that they cannot simply be omitted in general.

\item Moreover, this example gives an affirmative answer to the open question raised in \cite[p.~113]{AznayKato2023} concerning the converse of the implication from quasi semi-B-Fredholm to semi-B-Fredholm. It also shows that the Hilbert-space characterization in \cite[Theorem~2.6]{BerkaniSarih2001} does not extend to arbitrary Banach spaces. Indeed, $T_S$ is lower semi-B-Fredholm but is not pseudo-Fredholm, and hence is not quasi semi-B-Fredholm. Thus, a semi-B-Fredholm operator on a Banach space does not necessarily admit a decomposition into a semi-Fredholm operator and a nilpotent operator.
\end{enumerate}
\end{remark}

The next result shows that the analytic core, the quasinilpotent part and the kernels and ranges of all powers do not determine pseudo-Fredholmness on Banach spaces.

\begin{theorem}\label{th:obstruction}
There exist a complex Banach space $X$ and two operators $T_0,T_1\in\B(X)$ such that for every $n\ge1$,
\begin{equation}\label{eq:coincide}
\cc(T_0)=\cc(T_1),\qquad
\qn(T_0)=\qn(T_1),\qquad
\Ran(T_0^n)=\Ran(T_1^n),\qquad
\Ker(T_0^n)=\Ker(T_1^n),
\end{equation}
but $T_0$ is pseudo-Fredholm and $T_1$ is not. Both operators satisfy the three closedness conditions
\[
\cc(T_i)\text{ closed},\quad \cc(T_i)+\qn(T_i)\text{ closed},\quad
\Ran(T_i)+\qn(T_i)\text{ closed}.
\]
\end{theorem}
\begin{proof}
Let $E=\ell^1\oplus_1\ell^2$ and let $S\in\B(E)$ be the surjective
operator from Example~\ref{ex:separable}. Put $F=\ell^2(\mathbb N_0)$ and define
$J:F\to E$ by $Jy=(0,y)$. Then $SJ=0$.
Let $Q\in\B(F)$ be the weighted shift
\[
Qe_k=\frac{1}{k+1}e_{k+1}\quad(k\ge0).
\]
It is injective, compact and quasinilpotent, with $\|Q^n\|=1/n!$. On $X=E\oplus_1F$, define
\[
T_0=\begin{pmatrix}S&0\\0&Q\end{pmatrix},\qquad
T_1=\begin{pmatrix}S&J\\0&Q\end{pmatrix}.
\]
Since $E$ and $F$ are separable, $X$ is separable.
The operator $T_0$ is pseudo-Fredholm because $S$ is surjective and $Q$ is quasinilpotent.

Since $SJ=0$, for every $n\ge1$,
\[
T_1^n(x,y)=(S^nx+JQ^{n-1}y,Q^ny).
\]
The injectivity of $Q^n$ and surjectivity of $S^n$ imply
\[
\Ker(T_1^n)=\Ker(S^n)\oplus\{0\}=\Ker(T_0^n),\qquad
\Ran(T_1^n)=E\oplus\Ran(Q^n)=\Ran(T_0^n).
\]
Both operators induce $Q$ on the quotient $X/(E\oplus\{0\})$. Since $\cc(Q)=\{0\}$, their analytic cores are contained in $E\oplus\{0\}$. Conversely, $\cc(S)=E$ since $S$ is surjective. Hence
\[
\cc(T_0)=\cc(T_1)=E\oplus\{0\}.
\]
Moreover, $\|Q^ny\|^{1/n}\to0$ and $\|JQ^{n-1}y\|^{1/n}\to0$ for every $y\in F$. The above formula for the powers gives
\[
\qn(T_0)=\qn(T_1)=\qn(S)\oplus F.
\]
Thus \eqref{eq:coincide} holds, and both operators satisfy
\[
\cc(T_i)=E\oplus\{0\},\qquad
\cc(T_i)+\qn(T_i)=\Ran(T_i)+\qn(T_i)=X.
\]

Suppose, toward a contradiction, that $T_1$ is pseudo-Fredholm, and write $X=M\oplus N$ for a generalized Kato decomposition. Since the quasinilpotent restriction has analytic core zero and $\qn(T_1|_M)\subseteq \cc(T_1|_M)$, we have
\[
X=\cc(T_1)+\qn(T_1)=\cc(T_1|_M)\oplus N.
\]
As $X=M\oplus N$, it follows that $M=\cc(T_1)=E\oplus\{0\}$. Hence $N$ is the graph of some bounded map $V:F\to E$, since the second-coordinate projection is an isomorphism from $N$ onto $F$. Its invariance gives
\begin{equation}\label{eq:sylvester}
SV+J=VQ.
\end{equation}
Write $V=(V_1,V_2)$ with $V_1:F\to\ell^1$ and $V_2:F\to\ell^2$, and denote by $P_{\mathrm e}:\ell^1\to\ell^1$ the even-coordinate extraction. Taking the second $E$-coordinate in \eqref{eq:sylvester} yields
\[
I_F=V_2Q-qP_{\mathrm e}V_1.
\]
Since $\ell^2$ is reflexive and $\ell^1$ has the Schur property, $V_1$ is compact. Hence both $V_2Q$ and $qP_{\mathrm e}V_1$ are compact, which contradicts the non-compactness of $I_F$. Thus $T_1$ is not pseudo-Fredholm.
\end{proof}

\begin{corollary}
On general complex Banach spaces, even on separable ones, pseudo-Fredholmness cannot be characterized solely by the closedness or complementability of sums and intersections of $\cc(T)$, $\qn(T)$, $\Ran(T^n)$ and $\Ker(T^n)$, $n\ge1$.
\end{corollary}
\begin{proof}
By Theorem~\ref{th:obstruction}, these subspaces are the same for $T_0$ and $T_1$, although $T_0$ is pseudo-Fredholm and $T_1$ is not. Thus any condition involving only their sums, intersections, closedness or complementability cannot distinguish between the two operators.
\end{proof}

\section{Independence of the three closedness conditions}\label{sec:independence}
We now show that none of the three conditions in Corollary~\ref{thm:characterization} can be omitted. Two of the required examples are already available.

For the bilateral shift $B$ from Section~\ref{sec:mbekhta_questions}, Propositions~\ref{prop:bil} and~\ref{prop:bilcores} give
\[
 \cc(B)+\qn(B)=\Ran(B)+\qn(B)=\ell^2(\mathbb Z),
\]
while $\cc(B)$ is not closed.

Next consider the block operator $A=\begin{pmatrix}U&C\\0&0\end{pmatrix}$ from Example~\ref{rem:approx-spectrum-counterexample}. We proved that $\qn(A)=\{0\}$ and that $\Ran(A)$ is not closed. Moreover,
$\Ran(A^n)\subseteq U^{n-1}H\oplus\{0\}$ for all $n\ge1$. Since $\bigcap_{n\ge1}U^{n-1}H=\{0\}$, we have $\cc(A)\subseteq \Ran^\infty(A)=\{0\}$. Thus both $\cc(A)$ and $\cc(A)+\qn(A)$ are closed, whereas $\Ran(A)+\qn(A)$ is not.

We now construct a unilateral weighted shift for which $\cc(T)$ and
$\Ran(T)+\qn(T)$ are closed, but $\cc(T)+\qn(T)$ is not. We arrange
the weights in stages, each starting with small weights and ending with a
long string of weights equal to $1$.

Let $H=\ell^2(\mathbb N_0)$ with canonical orthonormal basis $(e_j)_{j\ge0}$. Define
\[
 b_n=\sum_{\nu=1}^{n-1}\nu(\nu+1)=\frac{(n-1)n(n+1)}3,\qquad n\ge1.
\]
Thus $b_1=0$, $b_2=2$, $b_3=8$, $b_4=20$, and $b_{n+1}-b_n=n(n+1)$. Stage $n$ consists of the weights with indices $b_n\le j<b_{n+1}$. In this stage, write $n$ copies of each of the numbers
\[
2^{-n},\ 2^{-(n-1)},\ldots,\ 2^{-1},\ 1,
\]
in that order. Equivalently, for every $j\ge0$, choose the unique $n$ with $b_n\le j<b_{n+1}$ and put
\begin{equation}\label{eq:weights}
 w_j=2^{-n+\lfloor(j-b_n)/n\rfloor},\qquad Te_j=w_je_{j+1}.
\end{equation}
For use in the proof, the same definition can be written
\begin{equation}\label{eq:levelweights}
 w_{b_n+(n-m)n+s}=2^{-m},\quad 0\le m\le n,\quad0\le s<n.
\end{equation}
Indeed, writing $j-b_n=qn+s$, with $0\le s<n$, and putting $m=n-q$, we obtain \eqref{eq:levelweights}.
\[
\begin{array}{c|c|l}
\text{stage}&\text{indices}&\text{weights in order}\\\hline
1&0,1&\tfrac12,1\\
2&2,\ldots,7&\tfrac14,\tfrac14,\tfrac12,\tfrac12,1,1\\
3&8,\ldots,19&\tfrac18,\tfrac18,\tfrac18,\tfrac14,\tfrac14,\tfrac14,
\tfrac12,\tfrac12,\tfrac12,1,1,1
\end{array}
\]
For instance, $w_{12}=2^{-3+\lfloor4/3\rfloor}=1/4$.

\begin{theorem}\label{thm:main}
The operator $T$ defined by \eqref{eq:weights} is an injective contraction satisfying
\[
 \cc(T)=\{0\},\qquad \Ran(T)+\qn(T)=H,\qquad
 \overline{\qn(T)}=H\ne \qn(T).
\]
Moreover, $\|T^r\|=1$ for every $r\ge1$ and $T^rx\to0$ for every $x\in H$. Its range is not closed, and it is not pseudo-Fredholm.
\end{theorem}
\begin{proof}
All weights lie in $(0,1]$. For $x=\sum_{j\ge0}x_je_j$, orthogonality of the vectors $Te_j$ gives
\[
 \|Tx\|^2=\sum_{j\ge0}|w_j|^2|x_j|^2\le\sum_{j\ge0}|x_j|^2=\|x\|^2.
\]
Therefore $T$ is a contraction. Since some weights are equal to $1$, we have $\|T\|=1$. All the weights are strictly positive, so $\Ker(T)=\{0\}$. For every $r\ge1$,
\begin{equation}\label{eq:shiftpowers}
 T^re_j=\left(\prod_{i=0}^{r-1}w_{j+i}\right)e_{j+r}.
\end{equation}
In particular every vector in $\Ran(T^r)$ has its first $r$ coordinates zero. If $x\in\bigcap_{r\ge1}\Ran(T^r)$, then its first $r$ coordinates vanish for every $r$, and hence $x=0$. Thus
\[
 \Ran^\infty(T)=\{0\},\qquad \cc(T)=\{0\},
\]
where the second equality follows from $\cc(T)\subseteq \Ran^\infty(T)$.
Let us next show that $T$ is not quasinilpotent. The final $n$ weights in stage $n$ equal $1$. With
\begin{equation}\label{eq:fn}
 a_n=b_n+n^2,\qquad f_n=e_{a_n},
\end{equation}
the vectors $f_n$ have norm $1$ and
\[
 T^rf_n=e_{a_n+r},\qquad 0\le r\le n.
\]
Indeed, $w_{a_n+i}=1$ for $0\le i<n$. It follows that $\|T^rf_n\|=1$ for $1\le r\le n$. For a fixed $r\ge1$, choose $n\ge r$. Thus $\|T^r\|\ge1$, while $\|T^r\|\le\|T\|^r=1$. Consequently $\|T^r\|=1$ for every $r\ge1$, and hence $r(T)=1$.

To control the orbits of the basis vectors, for an integer $k\ge1$
we shall use the norms introduced in \cite[Proposition~2.3]{AznayChabbabi2026}. In order to estimate them, put $p_k(x)=\sup_{r\ge0}k^r\|T^rx\|$ and $q_k(j)=p_k(e_j)$, allowing the value $+\infty$ initially. In the next lemma, we shall use the following elementary fact. If $0<c_1\le\cdots\le c_L$, and $P_r=\prod_{i=1}^r c_i$, $P_0=1$, then
\begin{equation}\label{eq:productmax}
 \max_{0\le r\le L}P_r=\max\{1,P_L\}.
\end{equation}

\Needspace{12\baselineskip}
\begin{lemma}\label{lem:orbit}
Fix an integer $k\ge2$ and set $p=1+\lceil2\log_2 k\rceil$. Then:
\begin{enumerate}[label=(\roman*)]
\item $q_k(b_n)=1$ for every $n\ge p$;
\item $q_k(b_n+(n-m)n+s+1)=1$ if $n\ge m\ge p$ and $0\le s<n$;
\item $q_k(j)<\infty$ for every fixed $j$, and $C_k:=\sup_{n\ge1}q_k(b_n)<\infty$.
\end{enumerate}
\end{lemma}
\begin{proof}
For (i), the factors $kw_j$ are nondecreasing within stage $n$, and their total product is
\[
 \prod_{j=b_n}^{b_{n+1}-1}kw_j
 =\prod_{m=0}^n(k2^{-m})^n
 =k^{n(n+1)}2^{-n^2(n+1)/2}
 =(k2^{-n/2})^{n(n+1)}\le1
\]
for $n\ge p$. Since the factors are nondecreasing on each stage,
\eqref{eq:productmax} gives
\[
 \prod_{j=b_\nu}^{b_\nu+t-1}kw_j\le1,
 \qquad \nu\ge p,\quad 0\le t\le b_{\nu+1}-b_\nu.
\]
For $r\ge0$, choose $N\ge n$ such that $b_N\le b_n+r<b_{N+1}$. Then
\[
 k^r\|T^re_{b_n}\|
 =\left(\prod_{\nu=n}^{N-1}\prod_{j=b_\nu}^{b_{\nu+1}-1}kw_j\right)
   \prod_{j=b_N}^{b_n+r-1}kw_j\le1.
\]
Consequently, $1\le q_k(b_n)\le1$.

For (ii), put $j=b_n+(n-m)n+s$ and $t=n-1-s$. Since
$m-1\ge2\log_2 k$, we have
\begin{align*}
 (k2^{-m})^t\prod_{u=0}^{m-1}(k2^{-u})^n
 &=(k2^{-m})^t\bigl(k^m2^{-m(m-1)/2}\bigr)^n\\
 &=(k2^{-m})^t\bigl(k2^{-(m-1)/2}\bigr)^{mn}\le1.
\end{align*}
The factors $kw_i$, $j+1\le i<b_{n+1}$, are nondecreasing.
Thus \eqref{eq:productmax} and (i) give
\[
 \prod_{i=j+1}^{j+r}kw_i\le1\quad(r\ge0),
 \qquad 1\le q_k(j+1)\le1.
\]

For (iii), fix $j$ and choose $N\ge p$ with $b_N\ge j$.
Put $L=b_N-j$. By (i),
\[
 q_k(j)\le\max_{0\le r\le L}\prod_{i=0}^{r-1}kw_{j+i}<\infty.
\]
Consequently,
\[
 C_k=\max\{1,q_k(b_1),\ldots,q_k(b_{p-1})\}<\infty.
\]
\end{proof}

For an arbitrary $z=\sum_{j\ge0}z_je_j$, orthogonality in \eqref{eq:shiftpowers} gives the useful inequality
\begin{equation}\label{eq:weightedPythagoras}
 p_k(z)^2\le\sum_{j\ge0}|z_j|^2q_k(j)^2.
\end{equation}
Indeed, by orthogonality, for every $r\ge0$ we have
$k^{2r}\|T^rz\|^2\le\sum_{j\ge0}|z_j|^2q_k(j)^2$.
Taking the supremum over $r$ gives \eqref{eq:weightedPythagoras}.

We can now determine the quasinilpotent part of $T$.
Lemma~\ref{lem:orbit}(iii) and \cite[Lemma~2.2]{AznayChabbabi2026} imply that every $e_j$ belongs to $\qn(T)$ (the estimate for $k=1$ follows from that for $k=2$). Thus $\qn(T)$ contains all finitely supported vectors and is dense in $H$.
Let us show that $\qn(T)\ne H$. The vectors $f_n$ in \eqref{eq:fn} are distinct basis vectors. Put
\[
 v=\sum_{n\ge1}\frac1n f_n\in H.
\]
For $r\ge1$, orthogonality of the images of distinct basis vectors gives
\[
 \|T^rv\|\ge\frac1r\|T^rf_r\|=\frac1r.
\]
On the other hand $\|T^rv\|\le\|v\|$. Hence
\[
 r^{-1/r}\le\|T^rv\|^{1/r}\le\|v\|^{1/r},
 \qquad \lim_{r\to\infty}\|T^rv\|^{1/r}=1.
\]
Thus $v\notin \qn(T)$, and $\qn(T)$ is not closed.
On the other hand, $T^rx\to0$ for every $x\in H$. This holds first on finitely supported vectors, which belong to $\qn(T)$. Given $x$ and $\varepsilon>0$, choose a finitely supported $u$ with $\|x-u\|<\varepsilon/2$. Then
\[
 \|T^rx\|\le\|T^r(x-u)\|+\|T^ru\|\le\varepsilon/2+\|T^ru\|<\varepsilon
\]
for all sufficiently large $r$. 

We now prove that $\Ran(T)+\qn(T)=H$. Let $y=\sum_{j\ge0}y_je_j\in H$.
We want to write $y=Tx+h$ with $x\in H$ and $h\in\qn(T)$.
Since the weights can be arbitrarily small, we shall lift only
some of the coordinates of $y$ through $T$.

To do this, we group the coordinates according to the weight immediately
preceding them. For $n\ge1$ and $0\le m\le n$, put
\[
 J_{n,m}=\{b_n+(n-m)n+s+1:0\le s<n\}\setminus\{b_{n+1}\},
\]
and let $E_{n,m}$ be the span of $\{e_j:j\in J_{n,m}\}$, with
$P_{n,m}$ its orthogonal projection. Thus
\[
 P_{n,m}y=\sum_{j\in J_{n,m}}y_je_j,
 \qquad w_{j-1}=2^{-m}\quad(j\in J_{n,m}).
\]
We exclude $b_{n+1}$ because it is the first index of the next stage.
For example, in stage $2$ the corresponding sets are
$J_{2,2}=\{3,4\}$, $J_{2,1}=\{5,6\}$ and $J_{2,0}=\{7\}$.
The indices $b_n$, together with the sets $J_{n,m}$, cover
$\mathbb N_0$ without repetition.

Fix $m\ge0$. Since the subspaces $E_{n,m}$ have disjoint coordinate
supports as $n$ varies, we have
\[
 \sum_{n\ge\max\{1,m\}}\|P_{n,m}y\|^2\le\|y\|^2.
\]
We may therefore choose an integer $N_m\ge\max\{1,m\}$ such that
\begin{equation}\label{eq:chosenNm}
 \sum_{n\ge N_m}\|P_{n,m}y\|^2\le 2^{-4(m+1)}.
\end{equation}
We make this choice for each $m$, independently of $k$.
The bound is chosen to compensate for the factor $2^m$ which appears
when we divide by a weight $2^{-m}$.

We take
\[
 y_R=\sum_{m\ge0}\sum_{n\ge N_m}P_{n,m}y
\]
and define
\[
 x=\sum_{m\ge0}\sum_{n\ge N_m}
       \sum_{j\in J_{n,m}}2^my_je_{j-1}.
\]
All the coordinate supports in these sums are disjoint. In particular,
using \eqref{eq:chosenNm}, we obtain
\begin{align*}
 \|x\|^2
 &=\sum_{m\ge0}2^{2m}\sum_{n\ge N_m}\|P_{n,m}y\|^2\\
 &\le\sum_{m\ge0}2^{2m-4(m+1)}
 =\frac1{12}<\infty.
\end{align*}
Hence $x\in H$. Moreover, for every $j\in J_{n,m}$,
\[
 T(2^my_je_{j-1})=2^my_jw_{j-1}e_j=y_je_j.
\]
It follows that $Tx=y_R$.

It remains to show that $h=y-Tx$ belongs to $\qn(T)$.
The coefficients of $h$ are exactly those which have not been used
in the construction of $x$. They occur either at the beginning $b_n$
of a stage or at indices in $J_{n,m}$ with $n<N_m$.

Fix an integer $k\ge2$ and let $p=1+\lceil2\log_2k\rceil$, as in
Lemma~\ref{lem:orbit}. We shall use the estimates of that lemma on
these remaining coordinates. At the stage beginnings we have
$q_k(b_n)\le C_k$. If $j\in J_{n,m}$ with $m\ge p$, then
$q_k(j)=1$ by Lemma~\ref{lem:orbit}(ii). Only the indices with
$m<p$ remain, and these form the finite set
\[
 F_k=\bigcup_{0\le m<p}
     \ \bigcup_{\max\{1,m\}\le n<N_m}J_{n,m}.
\]
By Lemma~\ref{lem:orbit}(iii), $q_k(j)<\infty$ for each $j\in F_k$.
Consequently,
\[
 \sum_{j\ge0}|h_j|^2q_k(j)^2
 \le \max\{1,C_k^2\}\|h\|^2
     +\sum_{j\in F_k}|h_j|^2q_k(j)^2<\infty.
\]
By \eqref{eq:weightedPythagoras}, $p_k(h)<\infty$.
This holds for every $k\ge2$, and also for $k=1$ because $T$ is a
contraction. The criterion in \cite[Lemma~2.2]{AznayChabbabi2026}
therefore gives $h\in\qn(T)$.

We have written the arbitrary vector $y\in H$ as
$y=Tx+h$, with $x\in H$ and $h\in\qn(T)$.
Thus $\Ran(T)+\qn(T)=H$.

Finally, we prove that $T$ is not pseudo-Fredholm.
Suppose $H=M\oplus N$ were a generalized Kato decomposition of $T$, with $S=T|_M$ semi-regular and $Q=T|_N$ quasinilpotent. Since $T$ is injective, $S$ is injective; its closed range therefore makes $S$ bounded below. For some $c>0$, $\|S^rm\|\ge c^r\|m\|$, so $\qn(S)=0$.
Let $P$ be the bounded projection onto $M$ along $N$. Invariance of both summands gives $PT=TP$. For $z\in \qn(T)$,
\[
 \|S^rPz\|=\|PT^rz\|\le\|P\|\|T^rz\|,
\]
so $Pz\in \qn(S)=0$. Since $\qn(T)$ is dense, $P=0$; hence $M=0$. This would make $T=Q$ quasinilpotent, contrary to $\|T^r\|=1$. Thus $T$ is not pseudo-Fredholm.
Finally, $\|Te_{b_n}\|=2^{-n}\to0$ while $\|e_{b_n}\|=1$. Injectivity and the bounded inverse theorem therefore show that $\Ran(T)$ is not closed. This completes the proof.
\end{proof}

The preceding constructions establish independence of the three closedness conditions.
\begin{proposition}\label{prop:independent}
For each of $\cc(T)$, $\cc(T)+\qn(T)$ and $\Ran(T)+\qn(T)$, there is a bounded operator on a separable complex Hilbert space for which that subspace is not closed but the other two are closed.
\end{proposition}
\begin{proof}
Use the bilateral shift $B$ for the first condition, the unilateral shift in Theorem~\ref{thm:main} for the second, and the block operator $A$ in Example~\ref{rem:approx-spectrum-counterexample} for the third.
\end{proof}

\begin{center}
\begin{tabular}{lccc}
\toprule
Operator & $\cc$ closed & $\cc+\qn$ closed & $\Ran+\qn$ closed\\
\midrule
Bilateral shift $B$ & No & Yes & Yes\\
Unilateral shift $T$ & Yes & No & Yes\\
Block operator $A$ & Yes & Yes & No\\
\bottomrule
\end{tabular}
\end{center}

The same examples also settle the independence of the Banach-space assumptions.

\begin{corollary}
None of the assumptions \textup{(C1)--(C4)} in
Theorem~\ref{th:suff} can be omitted.
\end{corollary}
\begin{proof}
For \textup{(C1)}, consider the bilateral shift $B$ from
Section~\ref{sec:mbekhta_questions}. Since
$\Ker(B)=\{0\}$ and
\[
\cc(B)+\qn(B)=\Ran(B)+\qn(B)=H,
\]
the other three conditions hold, but $B$ is not
pseudo-Fredholm.

For \textup{(C2)}, the unilateral shift $T$ in
Theorem~\ref{thm:main} satisfies
\[
\cc(T)=\Ker(T)=\{0\},\qquad
\Ran(T)+\qn(T)=H,
\]
but is not pseudo-Fredholm. Thus \textup{(C1), (C3)}
and \textup{(C4)} are not sufficient.

For \textup{(C3)}, take the operator $T_S$ from
Example~\ref{ex:separable}. It satisfies
\textup{(C1), (C2)} and \textup{(C4)}, but is not
pseudo-Fredholm. Indeed,
\[
\Ker(T_S)\cap\cc(T_S)=\Ker(S)\oplus\{0\}
\]
is not complemented in $E\oplus_1 E$, since a projection onto it
would give a bounded projection onto $\Ker(S)$ in $E$.

Finally, the block operator $A$ from
Example~\ref{rem:approx-spectrum-counterexample}
satisfies
\[
\cc(A)=\qn(A)=\{0\},
\]
while $\Ran(A)$ is not closed. Thus \textup{(C1), (C2)}
and \textup{(C3)} hold, but $A$ is not
pseudo-Fredholm. Hence \textup{(C4)} cannot
be omitted either.
\end{proof}

\end{document}